\documentclass[11pt,reqno]{amsart}
\usepackage[margin=1in]{geometry}
\usepackage{amsmath,amssymb,amsthm,graphicx,amsxtra, setspace}
\usepackage[utf8]{inputenc}
\usepackage{mathrsfs}
\usepackage{hyperref}
\usepackage{upgreek}
\usepackage{mathtools}
\usepackage[mathcal]{euscript}
\usepackage{xcolor}
\usepackage{upref}
\usepackage{multirow}
\allowdisplaybreaks

\usepackage[pagewise]{lineno}

\newtheorem{theorem}{Theorem}[section]
\newtheorem{lemma}[theorem]{Lemma}
\newtheorem{proposition}[theorem]{Proposition}

\newtheorem{corollary}[theorem]{Corollary}
\newtheorem{definition}[theorem]{Definition}
\newtheorem{example}[theorem]{Example}
\newtheorem{remark}[theorem]{Remark}

\newtheorem{hypothesis}[theorem]{Hypothesis}

\let\originalleft\left
\let\originalright\right
\renewcommand{\left}{\mathopen{}\mathclose\bgroup\originalleft}
\renewcommand{\right}{\aftergroup\egroup\originalright}

\def\N{\mathbb{N}}

\def\u{\boldsymbol{\xi}}
\def\v{\boldsymbol{\zeta}}
\def\w{\boldsymbol{\Upsilon}}
\def\f{\boldsymbol{\mathfrak{h}}}

\def\h{\varpi}
\def\Vbb{\mathbb{V}}
\def\Lbb{\mathbb{L}}
\def\Hbb{\mathbb{H}}

\def\Ocal{\mathcal{D}}
\def\Pcal{\mathcal{P}}
\def\Acal{\mathcal{A}}
\def\Hcal{\mathcal{H}}
\def\Mrm{\mathrm{M}}
\def\Frm{\mathrm{F}}
\def\Rrm{\boldsymbol{\mathfrak{R}}}
\def\drm{\mathrm{d}}
\def\Urm{\boldsymbol{\mathcal{X}}}
\def\Wrm{\mathcal{W}}
\def\Grm{\boldsymbol{\Phi}}
\def\Lrm{\mathrm{L}}
\def\Qrm{\mathrm{G}}

\def\Pbb{\mathbb{P}}
\def\Ebb{\mathbb{E}}
\def\Arm{\boldsymbol{\mathcal{E}}}

\def\B{\mathcal{B}}
\def\J{\mathcal{J}}
\def\K{\mathcal{K}}
\def\d{\mathrm{d}}
\def\C{\mathrm{C}}
\def\wi{\widetilde}
\def\Tr{\mathrm{Tr}}
\def\X{\mathbb{X}}

\def\W{\mathrm{W}}

\numberwithin{equation}{section}

\newcommand{\R}{\mathbb{R}}

\renewcommand{\d}{\/\mathrm{d}\/}

\newcommand{\Addresses}{{% additional braces for segregating \footnotesize
		\footnote{
			\noindent \textsuperscript{1}Center for Mathematics and Applications (NOVA Math), NOVA School of Science and Technology (NOVA FCT),	Portugal.\par\nopagebreak
			\noindent 
			\textit{e-mail:} \texttt{Kush Kinra: kushkinra@gmail.com, k.kinra@fct.unl.pt.}
			
			\noindent \textsuperscript{*}Corresponding author.
			
			\textit{Key words:} Stochastic continuous data assimilation, stochastic third-grade fluids, additive and multiplicative noise, Foias-Prodi estimates in expected value.
			
			Mathematics Subject Classification (2020): Primary 60H15, 35R60, 76B75; Secondary  93C20, 60H30, 37C50. 
			
			This work is funded by national funds through the FCT - Fundação para a Ciência e a Tecnologia, I.P., under the scope of the projects UID/297/2025 and UID/PRR/297/2025 (Center for Mathematics and Applications - NOVA Math).  K. Kinra wish to thank Prof. Fernanda Cipriano for suggesting this problem.
			
}}}

\begin{document}
		%	\linenumbers
	
	\title[Continuous data assimilation for stochastic third-grade fluid equations]{On Continuous Data Assimilation for a class of 2D and 3D stochastic non-Newtonian fluids of differential type
		\Addresses}
	
	\author[Kush Kinra]
	{Kush Kinra\textsuperscript{1*}}

	% Abstract
	\begin{abstract}
		Continuous data assimilation (CDA) techniques, most notably the nudging approach proposed by Azouani, Olson, and Titi (AOT), have been shown to be very successful in deterministic frameworks for achieving long-time synchronization between an approximate state and true state. In this note, we develop and study a CDA scheme for a class of stochastic non-Newtonian fluids, namely third-grade fluids, subject to either additive or multiplicative Gaussian stochastic forcing in both two- and three-dimensional settings. We establish sufficient criteria on the nudging gain and the observational mesh size that guarantee convergence of the assimilated state toward the underlying stochastic solution. Convergence is proved in the mean-square sense, and, in the case of additive noise, we further obtain almost sure (pathwise) convergence.
	\end{abstract}
	
	\maketitle

%  \tableofcontents

\section{Introduction}

In this note, we examine the continuous data assimilation problem for a class of stochastic non-Newtonian fluid models, specifically the stochastic third-grade fluid equations (TGFEs). It is important to note that, although much of the existing literature is devoted to Newtonian fluids described by the Navier–Stokes equations, many industrial and biological flows do not obey Newton’s law of viscosity and therefore cannot be adequately modeled within this framework. Such fluids often exhibit complex rheological properties, including shear-thinning, shear-thickening, and viscoelastic effects, which classical Newtonian models fail to capture. Consequently, more advanced mathematical descriptions are necessary to accurately represent their dynamics and predict their behavior in realistic environments.

In recent years, TGFEs have been employed in several numerical and simulation studies to better understand the dynamics of nanofluids (see, for example, \cite{PP19, RHK18}). Nanofluids are engineered suspensions of nanoparticles in a base fluid (such as water, oil, or ethylene glycol) and are known to possess enhanced thermal conductivity relative to the base fluid alone. This property makes them particularly important in technological and microelectronic applications. For these reasons, a rigorous mathematical analysis of TGFEs is essential for gaining insight into the behavior of such complex fluids.

\medskip
Let $\Ocal \subset \mathbb{R}^d$, $d=2,3$, be a bounded domain with a smooth boundary $\partial \Ocal$. The velocity field of the fluid is denoted by $\u(x,t) \in \mathbb{R}^d$, where $x \in \Ocal$ and $t \geq 0$, while $p(x,t) \in \mathbb{R}$ represents the corresponding pressure. An external forcing term is given by $\f(x) \in \mathbb{R}^d$. In this note, we study the following stochastic incompressible fluid system:
\begin{equation}\label{1}
	\left\{
	\begin{aligned}
		\drm \u -\nu \Delta\u \drm t + (\u\cdot\nabla)\u \drm t - \alpha \text{div}((\Arm(\u))^2) \drm t  & - \beta\text{div}(|\Arm(\u)|^2\Arm(\u)) \drm t +\nabla p \drm t \\
		&=\f \drm t + \Grm(t,\u)\drm\Wrm, && \text{ in } \ \Ocal\times(0,\infty), \\ \nabla\cdot\u&=0, \ &&\text{ in }  \Ocal\times[0,\infty), \\
		\u&=\mathbf{0},&& \text{ on } \ \partial\Ocal\times[0,\infty), \\
		\u(0)&=\u_0, && \text{ in } \ \Ocal,
	\end{aligned}
	\right.
\end{equation} 
where, $\Arm(\u)=\nabla \u+(\nabla \u)^T$, $\Wrm$ denotes a $\Qrm$-Wiener process, and $\Grm(\cdot,\cdot)$ represents the noise coefficient (see Subsection \ref{Sto-set} for further details). The parameter $\nu>0$ corresponds to the viscosity, while $\alpha,\beta>0$ are material moduli. In the special case $\alpha=\beta=0$, the model reduces to the classical $d$-dimensional stochastic Navier–Stokes equations (SNSEs).

Throughout this note, we assume that the parameters satisfy a subset of physically relevant conditions (see \cite{FR80}), given by:
\begin{align}\label{third-grade-paremeters-res}
	\nu>0,\quad \beta>0,\quad \text{and}\quad |\alpha|<\sqrt{2\nu\beta}.
\end{align}
Under the condition \eqref{third-grade-paremeters-res}, the well-posedness of system \eqref{1} was established in \cite{yas-fer_JNS}, where the authors also proved the existence of an ergodic invariant measure.

Data assimilation seeks to integrate spatially discrete observations into physical or computational models. Its central motivation lies in the intrinsic limitations of deterministic modeling, which often remains incomplete because not all underlying physical mechanisms can be fully represented. In classical continuous data assimilation (CDA), as introduced by Daley \cite{daley1993atmospheric}, observational data are incorporated directly into the model as it evolves in time. A more recent formulation due to Azouani, Olson, and Titi (AOT) \cite{azouani2014continuous} interprets CDA as a feedback control, or nudging, mechanism rooted in control theory \cite{AbderrahimTiti2014}. This approach augments the model with a relaxation term that drives the solution toward the observed data, yielding an assimilated trajectory that converges asymptotically to the true solution. As a result, the AOT algorithm offers an effective framework for state estimation in dynamical systems with limited observational data.

The AOT methodology has been widely applied to numerous deterministic physical models; see, for instance, \cite{BB_RMS_2025,bessaih2022continuous,BP_SIMA_2021,CGJA,FGHMMW,JST,PAM,NKC_2025_Submit} and the references therein. In particular, the deterministic system corresponding to \eqref{1} has been studied in \cite{NKC_2025_Submit}. In contrast, the extension of AOT methodology to stochastic frameworks has received relatively little attention (see, e.g., \cite{BFLZ_2025_Arxiv,bessaih2015continuous,Blomker,KK_CDA_SCBF}). The authors in \cite{BFLZ_2025_Arxiv} have recently extended the deterministic framework to stochastic settings.

 For completeness, let us briefly discuss the general structure of the CDA algorithm. Assume that $\u=\u(t)$ denotes the unknown true state of a dynamical system, which evolves according to
\begin{align*}
	\drm \u(t) = \Frm(\u(t)) \, \drm t + \Grm(t, \u(t)) \, \drm \Wrm(t),
\end{align*}
 with an unknown initial data $\u(0)$. Here $\Wrm$ is a $\Qrm$-Wiener process  and $\Grm(\cdot,\cdot)$ is the noise coefficient. Observational data are assumed to be accessible only at a coarse spatial scale via an interpolant operator $\Rrm_{\h}(\u(t))$, where $\h>0$ represents the observation resolution. The CDA algorithm produces an approximate solution by solving the modified equation
 \begin{align*}
 	\drm \Urm(t) = \left[ \Frm(\Urm(t)) - \kappa \Rrm_{\h}(\Urm(t) - \u(t)) \right]\, \drm t + \Grm (t,\Urm(t)) \, \drm  \Wrm(t).
 \end{align*}
 where   $\kappa>0$ is the nudging parameter. Under appropriate conditions on the parameters $\kappa$ and $\h$, the nudged solution $\Urm(t)$ converges to the true state $\u(t)$ as $t \to +\infty$, independently of the initial conditions $\u(0)$ and $\Urm(0)$. Estimates that measure the discrepancy between $\Urm$ and $\u$ are commonly referred to as Foias-Prodi estimates and typically yield exponential decay in time when $\kappa$ and $\h$ are chosen within suitable ranges. Examples of physically meaningful interpolant operators $\Rrm_{\h}$ may be found in \cite{foias1991determining,jones1992determining, jones1993upper}.
 
 Establishing convergence in the presence of stochastic forcing is more subtle due to the impact of randomness. Such analysis may be carried out either \emph{pathwise}, by studying almost-sure convergence for individual noise realizations, or \emph{in expectation}, by examining mean-square or statistical convergence. In the case of additive noise, the stochastic terms cancel when considering the difference $\u(t)-\Urm(t)$, which makes it possible to prove exponential convergence along almost every trajectory as $t \to +\infty$, irrespective of the initial data $\u(0)$ and $\Urm(0)$. For multiplicative noise, however, no such cancellation occurs, and convergence is generally obtained only in expectation.

In this work, we study the long-time convergence of the nudged solution $\Urm$ associated with system \eqref{1} toward the corresponding solution $\u$ of the same system. Our analysis builds on the framework developed in \cite{BFLZ_2025_Arxiv}, where related techniques were introduced, drawing in particular on methods previously used in the investigation of invariant measures (see, for example, \cite{BZ_DCDS,FZ25}).

Our main results can be summarized as follows. Assuming a linear growth condition on $\Grm$, together with appropriate restrictions on the observational resolution $\h$ and the nudging parameter $\kappa$, we establish that
 \begin{enumerate}
 	\item [(i).] for additive as well as multiplicative noise (see Theorem \ref{MT-Critical})
 	\begin{align*}
 		 \mathbb{E} \left[\|\u(t) - \Urm(t)\|_2^2 \right]\;\to\; 0
 		\quad \text{exponentially fast as } t \to +\infty.
 	\end{align*} 
 \vskip 2mm
 \item  [(ii).]  for additive noise (see Theorem \ref{pathwise_data_ass})
 \begin{align*}
 	  \|\u(t) - \Urm(t)\|_2^2 \;\to\; 0
 	\quad \text{exponentially fast as } t \to +\infty, \;\;\Pbb\text{-a.s.}
 \end{align*} 
 \end{enumerate}

The remainder of the paper is organized as follows. In Section~\ref{sec:setting}, we discuss the necessary preliminaries and well-posedness results. Section~\ref{Sec3} is dedicated to establishing energy estimates for solutions of system \eqref{STGF} and its associated data assimilation model, both in probability and in expectation. The analysis of the continuous data assimilation (CDA) algorithm for system \eqref{1} is carried out in Section~\ref{Sec4} (see Theorem~\ref{MT-Critical}). Finally, the last section provides a pathwise analysis of the CDA algorithm in the presence of additive noise (see Theorems~\ref{pathwise_data_ass}).

%%%%%%%%%%%%%%%%%%%%%%%%%%%%%%%
\section{Mathematical setting}\label{sec:setting}
In this section, we collect the notation, functional spaces, operators, and hypotheses that will be employed in the subsequent analysis. In addition, we discuss the solvability results for the underlying system and associated data assimilated system.

Let $\mathcal{E}$ be a Banach space, and denote its topological dual by $\mathcal{E}^*$. The duality pairing between $\mathcal{E}$ and $\mathcal{E}^*$ is written as $\langle \cdot , \cdot \rangle_{\mathcal{E},\mathcal{E}^*}$, with the subscripts omitted whenever the meaning is clear. For a real Hilbert space $\Hcal$, the associated norm and inner product are denoted by $\|\cdot\|_{\Hcal}$ and $(\cdot,\cdot)_{\Hcal}$, respectively. %The standard norms on the classical Lebesgue and Sobolev spaces $\mathrm{L}^p(\Ocal)$ (resp. $\mathbb{L}^p(\Ocal)$) and $\mathrm{W}^{m,p}(\Ocal)$ (resp. $\mathbb{W}^{m,p}(\Ocal)$) are denoted by $\|\cdot\|_p$ and $\|\cdot\|_{\mathrm{W}^{m,p}}$ (resp. $\|\cdot\|_{\mathbb{W}^{m,p}}$), respectively.

\subsection{Functional spaces}
 
Let $\C_0^{\infty}(\Ocal;\R^d)$ denote the space of all infinitely differentiable, $\R^d$-valued functions with compact support in $\Ocal\subset\R^d$. We define
\begin{align*}
	\mathcal{V} &:= \{\u \in \C_0^{\infty}(\Ocal;\R^d) : \nabla \cdot \u = 0\},  &&	\Hbb :=  \overline{\mathcal{V}}^{\Lbb^2(\Ocal) = \mathrm{L}^2(\Ocal;\R^d)},\\
	\Vbb &:= \overline{\mathcal{V}}^{\Hbb^1(\Ocal) = \mathrm{H}^1(\Ocal;\R^d)},
	&& \widetilde{\Lbb}^{p} := \overline{\mathcal{V}}^{\Lbb^p(\Ocal) = \mathrm{L}^p(\Ocal;\R^d)}, \quad p \in (1,\infty), \quad p\neq 2.
\end{align*}

The space $\Hbb$ is endowed with the Hilbert space structure inherited from $\Lbb^2(\Ocal)$, with the corresponding norm and inner product denoted by $\|\cdot\|_2$ and $(\cdot,\cdot)$, respectively. By virtue of the Poincar\'e inequality on the bounded domain $\Ocal$, the space $\Vbb$ may be equipped with the norm $\|\u\|_{\Vbb} := \|\nabla \u\|_2$ for $\u \in \Vbb$. For $p \in (1,\infty)$ with $p \neq 2$, the space $\wi\Lbb^p$ is endowed with the Banach space structure induced by $\Lbb^p(\Ocal)$, and its norm is denoted by $\|\cdot\|_p$. The embeddings $\Vbb \hookrightarrow \Hbb \cong \Hbb^* \hookrightarrow \Vbb^*$ are continuous, and the embedding $\Vbb \hookrightarrow \Hbb$ is compact. Further details on the functional setting can be found in \cite{Temam_1984}.

Let	$ \X:=\mathbb{W}_0^{1,4}(\Ocal)\cap \Vbb$ with norm $\Vert	\cdot\Vert_{\X} :=\Vert	\cdot\Vert_{\mathbb{W}^{1,4}}+\|\cdot\|_{\Vbb}.$
Indeed,	let us	recall	that	$\mathbb{W}_0^{1,4}(\Ocal)$	endowed	with $\Vert	\cdot\Vert_{\mathbb{W}^{1,4}}$-norm	is a Banach	space,	where		$$\Vert	\w \Vert_{\mathbb{W}^{1,4}}^4=\int_{\Ocal}	\vert	\w(z)\vert^4 \drm z+\int_{\Ocal} \vert	\nabla	\w(z)\vert^4 \drm z.$$

\subsection{The Helmholtz-Hodge projection}
Let $\mathcal{P}_p : \Lbb^2(\Ocal)\cap \Lbb^p(\Ocal) \to \Hbb\cap\widetilde{\Lbb}^p$ be a bounded linear projection (\cite{Farwig+Kozono+Sohr_2007}) such that, for $2\leq p<\infty$, the adjoint map  $(\mathcal{P}_p)^{*}=\mathcal{P}_{p'}$, where $\frac{1}{p}+\frac{1}{p'}=1$ and 
\begin{align*}
	\mathcal{P}_{p'} : \Lbb^2(\Ocal)+ \Lbb^{p'}(\Ocal) \to \Hbb+\widetilde{\Lbb}^{p'}.
\end{align*}
		\begin{lemma}[{\cite{Kunstmann_2010}}]
	The projection $\mathcal{P}_{\frac43}$ has a linear and continuous extension $\widetilde{\mathcal{P}}_{\frac43}: \mathbb{H}^{-1}(\Ocal) + \mathbb{W}^{-1,\frac{4}{3}}(\Ocal) \to \X^*$. 
%	given by the restriction
%	\begin{align}
%		\widetilde{\mathcal{P}}_{\frac{4}{3}} \Xi  = \Xi |_{\X}, \;\; \text{ for } \; \; \Xi \in \X^*.
%	\end{align}
\end{lemma}

\begin{remark}
We will henceforth denote $\widetilde{\mathcal{P}}_{\frac43}$ by $\mathcal{P}$.
\end{remark}

\subsection{Linear operator}\label{opeA}
Let us  introduce  the linear operator defined by 
\begin{equation*}
	\Acal\u:=-\mathcal{P}\Delta\u,\ \u\in\X.
\end{equation*}
Remember that the operator $\Acal$ is a non-negative operator in $\Hbb$ and 
\begin{align*}%\label{2.7a}
	\left<\u, \Acal \u\right>=\|\nabla\u\|_{2}^2,\ \textrm{ for all }\ \u\in\X, \ \text{ so that }\ \|\Acal\u\|_{\X^{*}} \leq \|\u\|_{\X}.
\end{align*}

\begin{remark}
	For the bounded domain $\Ocal$,  we have 
	\begin{align}\label{poin}
		\lambda_1\|\u\|_{2}^2 \leq 	\|\u\|_{\Vbb}^2,
	\end{align}
where $\lambda_1$ is the first eigenvalue of the Stokes operator (\cite{KK_CDA_SCBF}).
\end{remark}

%\begin{remark}
%	The Stokes operator $\Abb\u:= -\mathcal{P}_{2}\Delta\u$ is a linear operator in $\Hbb$ with domain $\Drm(\Abb)=:\Hbb^2(\Ocal)\cap\Vbb$. It can be extended (and let us not change the notation) as a linear operator $\Abb:\Vbb\to \Vbb^*$ by
%	\[
%	\langle \u , \Abb\v \rangle:=
%	(\nabla \u, \nabla\v)\,,  \text{ for all }  \u,\v\in \Vbb.
%	\]
%	For the bounded domain $\Ocal$, the operator $\Abb$ is invertible and its inverse $\Abb^{-1}$ is bounded, self-adjoint and compact in $\Hbb$. Thus, using spectral theorem, the spectrum of $\Abb$ consists of an infinite sequence $0< \lambda_1\leq \lambda_2\leq\ldots\leq \lambda_k\leq \ldots,$ with $\lambda_k\to\infty$ as $k\to\infty$ of eigenvalues. Moreover, there exists an orthonormal basis $\{e_k\}_{k=1}^{\infty} $ of $\Hbb$ consisting of eigenfunctions of $\Abb$ such that $\Abb e_k =\lambda_ke_k$,  for all $ k\in\mathbb{N}$. In addition, we have 
%		\begin{align}\label{poin}
%	\lambda_1\|\u\|_{2}^2 \leq 	\|\u\|_{\Vbb}^2.
%	\end{align}
%\end{remark}

\subsection{Bilinear operator}
We define the \emph{trilinear form} $b(\cdot,\cdot,\cdot):\X\times\X\times\X\to\R$ by $$b(\u,\v,\w)=\int_{\Ocal}(\u(z)\cdot\nabla)\v(z)\cdot\w(z)\drm z = \sum_{i,j=1}^d\int_{\Ocal}\u_i(z)\frac{\partial \v_j(z)}{\partial z_i}\w_j(z)\drm z.$$ If $\u, \v$ are such that the linear map $b(\u, \v, \cdot) $ is continuous on $\X$, the corresponding element of $\X^*$ is denoted by $\B(\u, \v)$. We also denote  $\B(\u, \u)= \mathcal{P} [(\u\cdot\nabla)\u]$.
An integration by parts yields  
\begin{equation}\label{b0}
	\left\{
	\begin{aligned}
		b(\u,\v,\v) &= 0, && \text{ for all }\ \u,\v \in\X,\\
		b(\u,\v,\w) &=  -b(\u,\w,\v), && \text{ for all }\ \u,\v,\w\in \X.
	\end{aligned}
	\right.\end{equation}

    \subsection{Nonlinear operators} 
We introduce two  nonlinear operators given by
\begin{equation*}
	\J(\u):= - \Pcal[\text{div}((\Arm(\u))^2)],\ \u\in\X,
\end{equation*}
and
\begin{equation*}
	\K(\u):= - \Pcal[\text{div}(|\Arm(\u)|^2\Arm(\u))] ,\ \u\in\X.
\end{equation*}
It can be easily seen that $\langle\u, \K (\u)\rangle = \frac{1}{2} \|\Arm(\u)\|_{4}^{4}$. 

	\begin{remark}
	1. From \cite[Equation (2.20)]{Hamza+Paicu_2007}, we have 
	\begin{align}\label{J1-J2}
		& |\alpha\left< \v_1 - \v_2,	\J(\v_1) - \J(\v_2) \right>| 
		\leq  \frac{|\alpha|}{2}\int_{\Ocal} |\Arm(\v_1-\v_2)|^2( |\Arm(\v_1)|+|\Arm(\v_2)|) \d x,
	\end{align}
	
	2. From \cite[Equation (2.13)]{Hamza+Paicu_2007}, we have 
	\begin{align}\label{K1-K2}
		&  \beta\left<\v_1 - \v_2,	\K(\v_1) - \K(\v_2) \right>
		 = \frac{\beta}{2}\int_{\Ocal}( |\Arm(\v_1)|^2-|\Arm(\v_2)|^2)^2 \d x + \frac{\beta}{2}\int_{\Ocal} |\Arm(\v_1-\v_2)|^2( |\Arm(\v_1)|^2+|\Arm(\v_2)|^2) \d x.
	\end{align}
\end{remark}

\begin{lemma}
In this lemma, we collect some inequalities that will be used later on:
	\begin{itemize}
		\item  We can find a positive constant $C_{S,d}$ (\cite[Subsection 2.4]{Kesavan_1989}) such that 
		\begin{align}\label{Sobolev-embedding3}
			\|\w\|_{\infty} \leq C_{S,d} \|\nabla\w\|_{4}, \quad \text{	for all } \;	\w\in	\mathbb{W}_0^{1,4}(\Ocal).
		\end{align}
		
		\item  We can find a positive constant $C_{K,d}$ (see \cite[Theorem 2 (ii)]{DiFratta+Solombrino_Arxiv}, a Korn-type inequality) such that 
		\begin{align}\label{Korn-ineq}
			\Vert	\nabla \w \Vert_{4}	\leq	C_{K,d}\Vert	\Arm(\w)	\Vert_4,\quad \text{	for all } \;	\w\in	\mathbb{W}_0^{1,4}(\Ocal).
		\end{align}		
		\item In view of \eqref{Sobolev-embedding3} and \eqref{Korn-ineq}, we can find a positive constant $\mathrm{N}_d$ such that
		\begin{align}\label{Sobolev+Korn}
			\|\w\|_{\infty} \leq \mathrm{N}_d \|\Arm(\w)\|_{4}, \quad \text{	for all } \;	\w\in	\mathbb{W}_0^{1,4}(\Ocal).
		\end{align}
	\end{itemize}
	
\end{lemma}

Throughout the article,  we denote by $C$   generic constant, which may vary from line to line.

\subsection{Stochastic setting}\label{Sto-set}
Let $(\Omega,\mathscr{F},\mathbb{P})$ be a complete probability space endowed with an increasing filtration $\{\mathscr{F}_t\}_{t\geq 0}$ of sub-$\sigma$-algebras of $\mathscr{F}$ satisfying the usual conditions.

\iffalse 
\subsubsection{$\Qrm$-Wiener process} 
We begin by recalling the definition and some properties of $\Qrm$-Wiener processes. Let $\Hcal$ be a separable Hilbert space.

\begin{definition}
	A stochastic process $\{\Wrm(t)\}_{t\geq 0}$ is called an \emph{$\Hcal$-valued, $\mathscr{F}_t$-adapted $\Qrm$-Wiener process} with covariance operator $\Qrm$ if:
	\begin{enumerate}
		\item[(i)] for every non-zero $\v\in \Hcal$, the process $\|\Qrm^{1/2}\v\|_{\Hcal}^{-1}(\Wrm(t),\v)$ is a standard one-dimensional Wiener process;
		\item[(ii)] for any $\v\in \Hcal$, the real-valued process $(\Wrm(\cdot),\v)$ is an $\mathscr{F}_t$-adapted martingale.
	\end{enumerate}
\end{definition}

The process $\{\Wrm(t)\}_{t\ge 0}$ is a $\Qrm$-Wiener process if and only if, for each $t$, one can represent the random field $\Wrm(\cdot)$ as 
$$\Wrm(\cdot,x) = \sum_{k=1}^{\infty} \sqrt{\mu_k}\,\boldsymbol{q}_k(x)\beta_k(\cdot),$$
where $\{\beta_k(\cdot)\}_{k\in\mathbb{N}}$ are independent one-dimensional Brownian motions on $(\Omega,\mathscr{F},\mathbb{P})$, and $\{\boldsymbol{q}_k\}_{k=1}^{\infty}$ is an orthonormal basis of $\Hcal$ satisfying $\Qrm\boldsymbol{q}_k = \mu_k \boldsymbol{q}_k$. If, in addition, $\Tr\Qrm = \sum_{k=1}^{\infty}\mu_k < \infty$, then $\Wrm(\cdot)$ is a Gaussian process in $\Hcal$ with 
$$\Ebb[\Wrm(t)] = 0, \qquad \mathrm{Cov}[\Wrm(t)] = t\Qrm,\qquad t\ge 0.$$
\fi

Assume that $\Hcal$ is a separable Hilbert space. Let $\{\Wrm(t)\}_{t\geq 0}$ be an \emph{$\Hcal$-valued, $\mathscr{F}_t$-adapted $\Qrm$-Wiener process} with covariance operator $\Qrm$. Define $\Hcal_0 := \Qrm^{1/2}\Hcal$, which becomes a Hilbert space when equipped with the inner product
$$(\u,\v)_0 = %\sum_{k=1}^{\infty}\frac{1}{\mu_k}(\u,\boldsymbol{q}_k)(\v,\boldsymbol{q}_k) = 
(\Qrm^{-1/2}\u,\Qrm^{-1/2}\v), \qquad \u,\v\in \Hcal_0,$$
where $\Qrm^{-1/2}$ denotes the pseudo-inverse of $\Qrm^{1/2}$.

Assume that $\mathcal{L}(\Hcal)$ is the space of bounded linear operators on $\Hcal$, and let $\mathcal{L}_{\Qrm}:=\mathcal{L}_{\Qrm}(\Hcal)$ denote the space of Hilbert-Schmidt operators from $\Hcal_0$ into $\Hcal$. Since $\Qrm$ is trace-class, the embedding $\Hcal_0 \hookrightarrow \Hcal$ is Hilbert–Schmidt, and $\mathcal{L}_{\Qrm}$ is itself a Hilbert space with 
%norm
%$$\|\Phi\|_{\mathcal{L}_{\Qrm}}^2 = \Tr(\Phi\Qrm\Phi^*),
%= \sum_{k=1}^{\infty} \|\Qrm^{1/2}\Phi^*\boldsymbol{q}_k\|_{\Hcal}^2,
%$$
%and 
inner product
$$(\Phi,\Psi)_{\mathcal{L}_{\Qrm}} 
= \Tr(\Phi\Qrm\Psi^*).
%= \sum_{k=1}^{\infty} \bigl(\Qrm^{1/2}\Psi^*\boldsymbol{q}_k,\Qrm^{1/2}\Phi^*\boldsymbol{q}_k\bigr).
$$
Further background can be found in \cite{DaZ}.

We now state the assumptions imposed on the noise coefficient $\Grm$.

\begin{hypothesis}\label{hyp-noise}
	Assume that $\{\Wrm(t)\}_{t\ge 0}$ is an $\Hbb$-valued $\Qrm$-Wiener process defined on the stochastic basis $(\Omega,\mathscr{F},\{\mathscr{F}_t\}_{t\geq 0},\mathbb{P})$. The coefficient $\Grm(\cdot,\cdot)$ satisfies:
	\begin{itemize}
		\item[(H.1)] $\Grm \in \C([0,T]\times \X;\,\mathcal{L}_{\Qrm}(\Hbb))$;
		
		\item[(H.2)] \textbf{(Linear growth condition)}  
		We can find two constants $K>0$ and $\widetilde{K} \geq 0$ such that for all $t\in [0,T]$ and $\u\in\Hbb$,
\begin{align*}
	\|\Grm(t,\u)\|_{\mathcal{L}_{\Qrm}}^2 	\leq K+\widetilde{K}\|\u\|_{2}^2.
\end{align*}
		
		\item[(H.3)] \textbf{(Lipschitz condition)}  
		We can find a constant $L\geq 0$ such that for any $t\in[0,T]$ and $\u_1,\u_2\in\Hbb$,
\begin{align*}
		\|\Grm(t,\u_1) - \Grm(t,\u_2)\|_{\mathcal{L}_{\Qrm}}^2 \leq L\,\|\u_1 - \u_2\|_{2}^2.
\end{align*}
	\end{itemize}
\end{hypothesis}

\subsection{Interpolant operators}\label{sec-inter}
When the initial velocity $\u_0$ is unknown, the data assimilation framework for system \eqref{1} employs a linear interpolant operator $\Rrm_{\h} : \mathbb{H}^1(\Ocal) \to \mathbb{L}^2(\Ocal)$ that approximates a given function using observations at a spatial resolution characterized by the length scale $\h>0$, and which satisfies the following properties:
 \begin{equation}\label{eqn-data-inter}
\left\| \u -  \Rrm_{\h}(\u) \right\|_{2}^2 \leq c_0 \h^2 \left\|\u \right\|_{\mathbb{H}^1}^2,
\end{equation}
 for every $\u \in \mathbb{H}^1(\Ocal).$ 
\begin{example}
	1). An example of an interpolant which is physically relevant and satisfies \eqref{eqn-data-inter} is given by finite volume elements  (see \cite{azouani2014continuous} and the therein references). 
	Specifically, let $\h > 0$ be given, and let 
	\[
	\Ocal = \bigcup_{j=1}^{n_{\h}} \Ocal_j,
	\]
	where the $\Ocal_j$ are disjoint subsets such that $\operatorname{diam}(\Ocal_j) \leq \h$ for 
	$j = 1, 2, \ldots, n_{\h}$. Then we set
	\[
	\Rrm_{\h}(\u)(x) = \sum_{j=1}^{N_{\h}} \overline{\u}_j\,\pmb{1}_{\Ocal_j}(x), \text{ where } \; \overline{\u}_j=\frac{1}{|\Ocal_j|}\int_{\Ocal_j} \u (x) \drm x.
	\]
	
	2). The orthogonal projection onto the Fourier modes, with wave numbers $k$ such that $|k| \leq \frac{1}{\h},$ is an another example of an interpolant operator which satisfies  the approximation property \eqref{eqn-data-inter}.
\end{example}

\subsection{Third-grade fluid equations and data assimilated equation}
In this section, we address the well-posedness of system \eqref{1} and introduce the corresponding data assimilation model, together with a discussion of its solvability. Applying the projection $\Pcal$ to system \eqref{1}, we obtain
\begin{equation}\label{STGF}
	\left\{
	\begin{aligned}
		\drm \u + \nu \Acal\u \drm t + \B(\u,\u) \drm t + \alpha \J(\u) \drm t + \beta\K(\u) \drm t  & =  \f \drm t + \Grm(t,\u)\drm\Wrm, &&t>0, \\ 
		\u(0)&=\u_0. && 
	\end{aligned}
	\right.
\end{equation} 
We recall the solvability of system \eqref{STGF} as established in \cite{yas-fer_JNS}.
\begin{definition}\label{def-StrongSolution}
	We say that the stochastic system \eqref{STGF} admits a \emph{strong solution} (in the probabilistic sense) if, for every stochastic basis $(\Omega,\mathscr{F},\{\mathscr{F}_t\}_{t\geq 0},\Pbb)$	and every $\Qrm$-Wiener process $\{\Wrm(t)\}_{t\geq 0}$ defined on this basis, there exists a progressively measurable process $\u : [0,+\infty)\times\Omega \to \Hbb$ such that, $\Pbb$-a.s.,
\begin{align*}
		\u(\cdot,\omega)\in \C([0,+\infty);\Hbb)\cap \Lrm_{\mathrm{loc}}^2(0,+\infty;\Vbb)\cap \Lrm_{\mathrm{loc}}^{4}(0,+\infty;\mathbb{W}_0^{1,4}(\Ocal)),
\end{align*}
	and for all $t\in[0,T]$ and all $\v\in\X$, the following identity holds $\Pbb$-a.s.:
	\begin{align*}
	& (\u(t),\v) 
	  \\ &  =(\u_0,\v)-\int_0^t\langle\v , \; \nu \Acal\u(s)+\B(\u(s),\u(s)) + \alpha\J(\u(s))  +\beta\K(\u(s)) - \f\rangle\d s 
	  + \int_0^t\left(\v , \; \Grm(s,\u(s))\d\W(s)\right).
\end{align*}
\end{definition}

 \begin{theorem}[{\cite{yas-fer_JNS}}]\label{thm-StrongSolution}
	Assume that condition \eqref{third-grade-paremeters-res} and Hypothesis \ref{hyp-noise} are satisfied. Also, let $\f \in \Vbb^*$ and $\u_0 \in \Hbb$. 
	Then, the system \eqref{STGF} admits a unique strong solution $\u$ in the sense of Definition \ref{def-StrongSolution}.
\end{theorem}

Let $\u$ be the solution of system \eqref{STGF}, and define $\Urm$ as the solution of the following equation, referred to as the data assimilation system:
\begin{equation}\label{STGF-CDA}
	\left\{
	\begin{aligned}
		\drm \Urm + \nu \Acal\Urm \drm t + \B(\Urm,\Urm) \drm t + \alpha \J(\Urm) \drm t + \beta\K(\Urm) \drm t  & =  \f \drm t  - \kappa \Pcal \Rrm_{\h}\left(\Urm-\u\right)  \drm t + \Grm(t,\Urm)\drm\Wrm  , \\ && \hspace{-5mm} t>0, \\ 
		\Urm(0)&=\Urm_0, && 
	\end{aligned}
	\right.
\end{equation} 
where, $\Rrm_{\h}$ is defined in Subsection \ref{sec-inter} and satisfies \eqref{eqn-data-inter}.

Note that if $\Rrm_{\h}$ satisfies condition \eqref{eqn-data-inter}, then the stochastic system \eqref{STGF-CDA} admits a unique strong (in the probabilistic sense) solution with the same regularity as that of system \eqref{STGF}. Hence, we have the following result:
\begin{theorem}
	Assume that condition \eqref{third-grade-paremeters-res} and Hypothesis \ref{hyp-noise} are satisfied. Also, let $\f \in \Vbb^*$ and $\Urm_0 \in \Hbb$. 
	Then, the system \eqref{STGF-CDA} admits a unique strong solution $\Urm$ with $\Pbb$-a.s. paths in 
	\begin{equation*}
		\C([0,+\infty), \Hbb) \cap \Lrm_{\mathrm{loc}}^2(0,+\infty;\Vbb)\cap \Lrm_{\mathrm{loc}}^{4}(0,+\infty;\mathbb{W}_0^{1,4}(\Ocal)),
	\end{equation*}
and for all $t\in[0,T]$ and all $\v\in\X$, the following identity holds $\Pbb$-a.s.:
\begin{align*}
	 (\Urm(t),\v) 
	& =(\Urm_0,\v)-\int_0^t\langle\v , \; \nu \Acal\Urm(s)+\B(\Urm(s),\Urm(s)) + \alpha \J(\Urm(s)) +\beta\K(\Urm(s)) - \f\rangle\d s
	\nonumber\\ 
	& \quad  + \kappa \int_0^t\left(\v , \Rrm_{\h}(\Urm(s)-\u(s))\right)\drm s
	 \quad  + \int_0^t\left(\v , \; \Grm(s,\Urm(s))\d\W(s)\right).
\end{align*}
\end{theorem}

\section{Estimates in probability and expectation}\label{Sec3}
In this section, we develop a collection of estimates in probability and expectation that are fundamental to the derivation of the main results of this article.

\subsection{Estimate in probability}
Let us first provide estimates in the probability which will be useful in the sequel.
\begin{lemma}\label{lem-prob-esti-2}
 Assume that condition \eqref{third-grade-paremeters-res} and Hypothesis \ref{hyp-noise} are satisfied. Then, the solution $\u$ of \eqref{STGF} satisfies:
		\begin{align}\label{eqn-Prob-Est-varpi=3}
			& \Pbb\left\{\sup_{t \geq T} \left( \|\u(t)\|^{2}_{2} + \frac{\beta}{4} \int_{0}^{t}\|\Arm(\u(s))\|^4_{4} \drm s - \|\u_0\|^{2}_{2}  - \Mrm t  \right) \geq R \right\}
			\leq 
			\begin{cases}
				e^{-  \frac{\nu\lambda_1}{4K} R}, & \text{ for } L=0; \\
				e^{- \min \left\{ \frac{\beta\lambda_1}{16L |\Ocal|}, \frac{\nu\lambda_1}{4K} \right\} R}, & \text{ for } L>0;
			\end{cases}
		\end{align}
		for all $T\geq 0$ and $R>0$, where $\Mrm$ is given by 
	\begin{align}\label{eqn-Mrm}
	\Mrm =	K  + \frac{1}{4\beta}\left(\frac{\widetilde{K}}{\lambda_1}+1\right)^2|\Ocal|  +  \frac{27\alpha^4 }{4\beta^3} |\Ocal| 
	+  \|\f\|^2_{\Vbb^{*}}.
\end{align}	
\end{lemma}
\begin{proof}
	Applying the It\^o formula to the process $\|\u(\cdot)\|_{2}^{2}$, together with $\eqref{b0}_1$, \eqref{poin}, and H\"older's and Young's inequalities, we obtain
	\begin{align*}
		&\|\u(t)\|^{2}_{2} + \frac{\beta}{2}\int_{0}^{t}\|\Arm(\u(s))\|^{4}_{4} \drm s
		\nonumber\\&
		= \|\u_0\|^{2}_{2} -2\nu\int_{0}^{t}\|\u(s)\|^2_{\Vbb} \drm s - \frac{\beta}{2}\int_{0}^{t}\|\Arm(\u(s))\|^{4}_{4} \drm s +   \int_{0}^{t}\|\Grm(s,\u(s))\|^2_{\mathcal{L}_{\Qrm}} \drm s
		\nonumber\\ 
		& \quad  +  2 \int_{0}^{t}\left(\Grm(s,\u(s))\drm \Wrm(s) , \u(s)\right) 
		 - \alpha  \int_{0}^{t}\int_{\Ocal}  [(\Arm(\u(s)))^2 : \Arm(\u(s)) ] \drm x \drm s
		   +  2 \int_{0}^{t}\left< \u( s),   \f  \right> \drm s
		\nonumber \\&
		\leq \|\u_0\|^{2}_{2}  -2\nu\lambda_1\int_{0}^{t}\|\u(s)\|^2_{2} \drm s  - \frac{\beta}{2}\int_{0}^{t}\|\Arm(\u(s))\|^{4}_{4} \drm s +  \int_{0}^{t}\|\Grm(s,\u(s))\|^2_{\mathcal{L}_{\Qrm}} \drm s 
		\nonumber\\ 
		& \quad + |\alpha|  \int_{0}^{t} \|\Arm(\u(s))\|_4^2  \|\Arm(\u(s))\|_2  \drm s
		 +  2 \int_{0}^{t}\|\f\|_{\Vbb^{*}} \|\u( s)\|_{\Vbb} \drm s 
		 +  2 \int_{0}^{t} \left(\Grm(s,\u(s))\drm \Wrm(s) , \u( s)\right) 
		\nonumber
		\\&
		\leq \|\u_0\|^{2}_{2}  -2\nu\lambda_1\int_{0}^{t}\|\u(s)\|^2_{2} \drm s  - \frac{\beta}{2}\int_{0}^{t}\|\Arm(\u(s))\|^{4}_{4} \drm s + K t + \widetilde{K} \int_{0}^{t}\|\u(s)\|^2_{2} \drm s 
		\nonumber\\ 
		& \quad + 2 |\Ocal|^{\frac{1}{4}} |\alpha|  \int_{0}^{t} \|\Arm(\u(s))\|_4^3 \drm s
		 +   \|\f\|^2_{\Vbb^{*}} t  
		 +  \int_{0}^{t} \|\u( s)\|^2_{\Vbb} \drm s 
		 +  2 \int_{0}^{t} \left(\Grm(s,\u(s))\drm \Wrm(s) , \u( s)\right) 
		\nonumber
		\\&
		\leq \|\u_0\|^{2}_{2} -2\nu\lambda_1\int_{0}^{t}\|\u(s)\|^2_{2} \drm s - \frac{\beta}{2}\int_{0}^{t}\|\Arm(\u(s))\|^{4}_{4} \drm s  + K t + \frac12\left(\frac{\widetilde{K}}{\lambda_1}+1\right) \int_{0}^{t}\|\Arm(\u(s))\|^2_{2} \drm s \nonumber\\ 
		& \quad +  |\Ocal|^{\frac{1}{4}} |\alpha|  \int_{0}^{t} \|\Arm(\u(s))\|_4^3 \drm s
		 +  \|\f\|^2_{\Vbb^{*}} t  
		+  2 \int_{0}^{t} \left(\Grm(s,\u(s))\drm \Wrm(s) , \u( s)\right) 
		\nonumber
		\\&
		\leq \|\u_0\|^{2}_{2} -2\nu\lambda_1\int_{0}^{t}\|\u(s)\|^2_{2} \drm s - \frac{\beta}{2}\int_{0}^{t}\|\Arm(\u(s))\|^{4}_{4} \drm s + K t + \frac12\left(\frac{\widetilde{K}}{\lambda_1}+1\right)|\Ocal|^{\frac12} \int_{0}^{t}\|\Arm(\u(s))\|^2_{4} \drm s 
		\nonumber\\ 
		& \quad +  |\Ocal|^{\frac{1}{4}} |\alpha|  \int_{0}^{t} \|\Arm(\u(s))\|_4^3 \drm s
		  +  \|\f\|^2_{\Vbb^{*}} t  
		 +  2 \int_{0}^{t} \left(\Grm(s,\u(s))\drm \Wrm(s) , \u( s)\right) 
		\nonumber
		\\&
		\leq \|\u_0\|^{2}_{2} -2\nu\lambda_1\int_{0}^{t}\|\u(s)\|^2_{2} \drm s + \Mrm t  
		+  \underbrace{2 \int_{0}^{t} \left(\Grm(s,\u(s))\drm \Wrm(s) , \u( s)\right) }_{=:\mathcal{M}(t)},
	\end{align*}
	where the constant $\Mrm >0$ is independent of time $t$, and given by \eqref{eqn-Mrm}. It follows that $\{\mathcal{M}(t)\}_{t\geq 0}$ forms a martingale, and its quadratic variation is given by
\begin{align*}%\label{PrEs2}
	& [\mathcal{M}](t) 
	\nonumber\\ 
	& \leq 4 \int_{0}^{t} \|\u(s)\|_2^2 \|\Grm(s,\u(s))\|^2_{\mathcal{L}_{\Qrm}} \drm s  
	 \leq 8 \int_{0}^{t} \|\u(s)\|_2^2 [\|\Grm(s,\u(s)) - \Grm(s,\boldsymbol{0})\|^2_{\mathcal{L}_{\Qrm}} + \|\Grm(s,\boldsymbol{0})\|^2_{\mathcal{L}_{\Qrm}}] \drm s  
	\nonumber\\ 
	& \leq 8 \int_{0}^{t}  \left[L\|\u(s)\|_2^4  + K\|\u(s)\|_2^2 \right] \drm s
	 \leq 8 \int_{0}^{t}  \left[\frac{L}{\lambda_1}\|\u(s)\|_{\Vbb}^4  + K\|\u(s)\|_2^2 \right] \drm s
	\nonumber\\ 
	& \leq  8 \int_{0}^{t}  \left[\frac{L}{2\lambda_1}\|\Arm(\u(s))\|_{2}^4  + K\|\u(s)\|_2^2 \right] \drm s
	 \leq  8 \int_{0}^{t}  \left[\frac{L|\Ocal| }{2\lambda_1}\|\Arm(\u(s))\|_{4}^4  + K\|\u(s)\|_2^2 \right] \drm s.
\end{align*}
This helps us to get 
\begin{align*}
	& \|\u(t)\|^{2}_{2} +  \frac{\beta}{4} \int_{0}^{t}\|\Arm(\u(s))\|^4_{4} \drm s -\|\u_0\|^{2}_{2} - \Mrm t
	\nonumber\\
	& \leq    \mathcal{M}(t) -   \int_{0}^{t}\left(2\nu\lambda_1\|\u(s)\|^2_{2} +\frac{\beta}{4}\|\Arm(\u(s))\|^{4}_{4}\right)\drm s
	\nonumber\\ 
	& =    \mathcal{M}(t) -  \frac{ \nu\lambda_1 }{4K} 8K   \int_{0}^{t}\|\u(s)\|^2_{2}\drm s  - \frac{\beta\lambda_1}{16L |\Ocal|} \frac{4L|\Ocal|}{\lambda_1} \int_{0}^{t}\|\u(s)\|^{4}_{4} \drm s	
	\nonumber\\ 
	& \leq    \begin{cases}
		\mathcal{M}(t) - \frac{\nu\lambda_1}{4K} [\mathcal{M}](t), & \text{ for } L=0;\\
		\mathcal{M}(t) - \min \left\{ \frac{\beta\lambda_1}{16L |\Ocal|}, \frac{\nu\lambda_1}{4K} \right\} [\mathcal{M}](t), & \text{ for } L>0,
	\end{cases}
\end{align*}
which together with the exponential martingale inequality \cite[Proposition 3.1]{Glatt-Holtz_2014_Arxiv}, implies that \eqref{eqn-Prob-Est-varpi=3} holds, thereby completing the proof.
\end{proof}

\subsection{Estimate in expectation}
Let us now provide estimates in the expectation which will be used later.

\begin{lemma}\label{lem-pth-moments}
Assume that condition \eqref{third-grade-paremeters-res} and Hypothesis \ref{hyp-noise} are satisfied. Then, the solution $\u$ of \eqref{STGF} satisfies
		\begin{align*}%\label{eqn-moments-4-u}
			& \sup_{t\geq 0}	\Ebb \left[\|\u(t)\|^{4}_{2}\right] \leq \|\u_0\|^{4}_{2} + C,
	\end{align*}
where $C>0$ is some constant independent of $t$.
\end{lemma}
\begin{proof}
	Using the It\^o formula on $\|\u(\cdot)\|_{2}^{4}$ in combination with $\eqref{b0}_1$, \eqref{poin}, and H\"older's and Young's inequalities gives
	\begin{align}\label{EE2}
		&\|\u(t)\|^{4}_{2} + 4\nu\int_{0}^{t}\|\u(s)\|_{2}^{2}\|\u(s)\|^2_{\Vbb} \drm s + 2\beta\int_{0}^{t}\|\u(s)\|_{2}^{2}\|\Arm(\u(s))\|^{4}_{4} \drm s
		\nonumber\\&
		= \|\u_0\|^{4}_{2} + 2  \int_{0}^{t}\|\u(s)\|_{2}^{2}\|\Grm(s,\u(s))\|^2_{\mathcal{L}_{\Qrm}} \drm s
		+  4 \int_{0}^{t} \|\u(s)\|_{2}^{2} \left(\Grm(s,\u(s))\drm \Wrm(s) , \u(s)\right) 
		\nonumber\\ 
		& \quad - 2 \alpha  \int_{0}^{t} \|\u(s)\|_{2}^{2} \left<  (\Arm(\u(s)))^2 , \Arm(\u(s)) \right> \drm s
		+  4 \int_{0}^{t} \|\u(s)\|_{2}^{2} \; \left< \u( s),   \f  \right> \drm s
		\nonumber\\
		& +  4 \int_{0}^{t} \|(\Grm(s,\u(s)))^{*}\u( s)\|^2_{2}\drm s
		\nonumber \\
		&
		\leq \|\u_0\|^{4}_{2} + 6  \int_{0}^{t} \|\u(s)\|_{2}^{2} \|\Grm(s,\u(s))\|^2_{\mathcal{L}_{\Qrm}} \drm s + 2 |\alpha|  \int_{0}^{t} \|\u(s)\|_{2}^{2} \|\Arm(\u(s))\|_4^2  \|\Arm(\u(s))\|_2  \drm s
		\nonumber\\ & 
		\quad  +  4 \int_{0}^{t} \|\u(s)\|_{2}^{2} \|\f\|_{\Vbb^{*}} \|\u( s)\|_{\Vbb} \drm s 
		  +  4 \int_{0}^{t} \|\u(s)\|_{2}^{2} \left(\Grm(s,\u(s))\drm \Wrm(s) , \u( s)\right) 
		\nonumber \\
		&
		\leq \|\u_0\|^{4}_{2}  + 6K \int_{0}^{t}\|\u(s)\|^2_{2} \drm s  + 6\widetilde{K} \int_{0}^{t}\|\u(s)\|^4_{2} \drm s  + 2 |\alpha|  \int_{0}^{t} \|\u(s)\|_{2}^{2} \|\Arm(\u(s))\|_4^2  \|\Arm(\u(s))\|_2  \drm s
		\nonumber\\ & 
		\quad  +  4 \int_{0}^{t} \|\u(s)\|_{2}^{2} \|\f\|_{\Vbb^{*}} \|\u( s)\|_{\Vbb} \drm s 
		+  4 \int_{0}^{t} \|\u(s)\|_{2}^{2} \left(\Grm(s,\u(s))\drm \Wrm(s) , \u( s)\right) 
		\nonumber \\
		&
		\leq \|\u_0\|^{4}_{2} + \beta\int_{0}^{t}\|\u(s)\|_{2}^{2}\|\Arm(\u(s))\|^{4}_{4} \drm s + t C_{K,\widetilde{K},\alpha,\beta,\lambda_1,|\Ocal|,\|\f\|_{\Vbb^{*}}}
		\nonumber\\
		& +  4 \int_{0}^{t} \|\u(s)\|_{2}^{2} \left(\Grm(s,\u(s))\drm \Wrm(s) , \u( s)\right),
	\end{align}
	where the constant $C := C_{K,\widetilde{K},\alpha,\beta,\lambda_1,|\Ocal|,\|\f\|_{\Vbb^{*}}} >0$ is independent of time $t$.
	
	Considering expectations in \eqref{EE2}, applying \eqref{poin}, and recalling that the stochastic integral defines a (local) martingale, we conclude that
	\begin{align*}
		&\Ebb \left[\|\u(t)\|^{4}_{2}\right]
		+  4\nu\lambda_1 \int_{0}^{t}\Ebb\left[\|\u(s)\|_{2}^{4}\right]\drm s 
		\leq \|\u_0\|^{2}_{2}  +      t C,
	\end{align*}
	which, using Gronwall’s inequality, yields
	\begin{align*}
		\Ebb \left[\|\u(t)\|^{4}_{2}\right] \leq \|\u_0\|^{4}_{2} e^{- 4\nu\lambda_1 t} + \frac{C}{4\nu\lambda_1} (1- e^{-4\nu\lambda_1 t}).
	\end{align*}
	Hence, we complete the proof.
\end{proof}

\begin{lemma}\label{lem-pth-moments-Urm}
	Assume that condition \eqref{third-grade-paremeters-res} and Hypothesis \ref{hyp-noise} are satisfied. Then, the solution $\Urm$ of \eqref{STGF-CDA} satisfies
		\begin{align*}%\label{eqn-moments-2p-Urm}
			& \sup_{t\geq 0}	\Ebb \left[\|\Urm(t)\|^{4}_{2}\right] \leq   C(1+\|\u_0\|^{4}_{2} + \|\Urm_0\|^{4}_{2}),
	\end{align*}
where $C>0$ is a constant independent of $t$.
\end{lemma}
\begin{proof}
	Note that system \eqref{STGF-CDA} contains an additional term, $-\kappa \Rrm_{\h}(\Urm-\u)$, compared to system \eqref{STGF}. This term can be estimated in the same way as in the proof of \cite[Lemma A.4]{BFLZ_2025_Arxiv}. Therefore, we omit the proof. Making use of Lemma \ref{lem-pth-moments}, one can complete the proof.
\end{proof}

\section{Continuous data assimilation}\label{Sec4}

We perform the convergence analysis in expectation, which ensures that the results are valid in the mean-square sense. Assuming suitable conditions on the observational resolution $\h$ and the nudging parameter $\kappa$, we can show that
\[
\mathbb{E}\bigl[\|\u(t)-\Urm(t)\|_2^2\bigr] \to 0
\quad \text{as } t \to +\infty.
\]
In data assimilation literature, the quantity $\mathbb{E}\bigl[\|\u(t)-\Urm(t)\|_2^2\bigr]$ is often referred to as a Foias-Prodi estimate in expectation. This result indicates that the control term $\kappa\,\Rrm_{\h}\bigl(\Urm(t)-\u(t)\bigr),$ when $\kappa$ and $\h$ are chosen appropriately, successfully drives the nudged solution $\Urm$ toward the reference solution $\u$ in the mean-square sense as time grows large.  

Our initial derivation is carried out for a general stopping time, with the specific choice of stopping time postponed to a subsequent stage of the analysis.

\begin{lemma}\label{lem-Difference}
	 Suppose that condition \eqref{third-grade-paremeters-res}, Hypothesis \ref{hyp-noise} and assumption \eqref{eqn-data-inter} are satisfied.  Assume that $\u$ and $\Urm$ are the unique solutions of  systems \eqref{STGF} and \eqref{STGF-CDA}, respectively. If $\kappa$ and $\h$ are such that 
		\begin{align}\label{condition-on-sigma-4}
			0  <  \kappa \leq   \frac{  \nu \varepsilon_0 }{c_0 \h^2 }
		\end{align}
		where $c_0$ is the constant appearing in \eqref{eqn-data-inter}, then for any $\delta\ge 0$, any stopping time $\tau$ and any $t \ge 0$:
		\begin{align}\label{eqn-Difference-3}
			\mathbb{E}\left[e^{\left(\frac{\kappa}{1+\delta} -  L \right)t\wedge \tau  -  \frac{27[\mathrm{N}_d]^4}{16\lambda_1\nu^3\varepsilon_0^3} \int_0^{t\wedge \tau}\|\Arm(\u(s))\|_{4}^4\drm s}\left\| \u(t\wedge \tau)-\Urm(t\wedge \tau) 
			\right\|_{2}^2\right] 
			&\leq \left\|\u_0 - \Urm_0 \right\|_{2}^2.
		\end{align}
\end{lemma}

\begin{proof}
	 We define $\w(\cdot):=\Urm(\cdot)-\u(\cdot)$, then $\w(\cdot)$ satisfies
	\begin{equation*}
		\left\{
		\begin{aligned}
			\drm \w(t) + \big[\nu \Acal\w(t)      &  + \left( \B(\Urm(t),\w(t)) + \B(\w(t),\u(t)) \right)  + \alpha\left( \J(\Urm(t)) - \J(\u(t)) \right)  \\  + \beta\left( \K(\Urm(t)) - \K(\u(t)) \right)  \big]\, \drm t  
			& = - \kappa \Pcal \Rrm_{\h}(\w(t)) + [\Grm(t,\Urm(t))-\Grm(t,\u(t))]\, \drm\Wrm(t), \;\; t>0  \\ 
			\w(0)& =\Urm_0 - \u_0.
		\end{aligned}
		\right.
	\end{equation*}
	By applying the It\^o formula to $\|\w(\cdot)\|_2^2$ and employing $\eqref{b0}_1$, we get, $\mathbb{P}$-a.s.,
	\begin{align}\label{Diff-1}
		& \|\w(t)\|^2_2 + 2\nu \int_0^t \|\w(s)\|^2_{\Vbb}\,\drm s + 2 \beta \int_0^t \left\langle\w(s) , \K(\Urm(s)) - \K(\u(s)) \right\rangle \drm s
		\nonumber\\
		&  =  \|\w(0)\|^2_2 - 2 \int_0^t \left\langle \w (s), \B(\w(s), \u(s)) \right\rangle \drm s   + 2 \kappa \int_0^t  \left( \w (s),  \Rrm_{\h}(\w(s)) \right) \drm s  
		\nonumber\\
		 &\quad - 2 \alpha \int_0^t \left\langle\w(s) , \J(\Urm(s)) - \J(\u(s)) \right\rangle \drm s
		 + \int_0^t \|\Grm(s,\Urm(s))-\Grm(s,\u(s))\|^2_{\mathcal{L}_{\Qrm}} \drm s 
		 \nonumber\\ 
		 & \quad  +  2 \int_0^t \big( \w (s), [\Grm(s,\Urm(s))-\Grm(s,\u(s))]\, \drm\Wrm(s) \big)
		\nonumber\\
		&  = \|\w(0)\|^2_2 - 2\kappa  \int_0^t \|\w(s)\|^2_2 \,\drm s - 2 \int_0^t  b( \w(s), \u(s), \w (s) ) \drm s    + 2 \kappa \int_0^t  \left( \w (s),  \Rrm_{\h}(\w(s)) - \w(s) \right) \drm s  
		\nonumber\\ & 
		\quad - 2 \alpha \int_0^t \left\langle\w(s) , \J(\Urm(s)) - \J(\u(s)) \right\rangle \drm s
	+ \int_0^t \|\Grm(s,\Urm(s))-\Grm(s,\u(s))\|^2_{\mathcal{L}_{\Qrm}} \drm s  
		\nonumber\\ 
	& \quad  +  2 \int_0^t \big( \w (s), [\Grm(s, \Urm(s))-\Grm(s, \u(s))]\, \drm\Wrm(s) \big).
	\end{align}

For $\varepsilon_0:=1-\sqrt{\frac{\alpha^2}{2\beta\nu}}\in(0,1)$ (it is possible to define by \eqref{third-grade-paremeters-res}), from \eqref{J1-J2} and H\"older's inequality, we get
\begin{align}\label{Diff-2}
	 |2\alpha\left<	\w, \J(\Urm) - \J(\u) \right>| 
	 & \leq \nu(1-\varepsilon_0) \|\Arm(\w)\|_{2}^2 + \frac{\alpha^2}{2\nu(1-\varepsilon_0)}  \int_{\Ocal} |\Arm(\w)|^2 ( |\Arm(\Urm)|^2+|\Arm(\u)|^2) \drm x
	\nonumber\\ & = 2\nu(1-\varepsilon_0)\|\w\|_{\Vbb}^2 + \beta(1-\varepsilon_0)  \int_{\Ocal} |\Arm(\w)|^2( |\Arm(\Urm)|^2+|\Arm(\u)|^2) \drm x.
\end{align}

From \eqref{K1-K2}, we have 
\begin{align}\label{Diff-3}
	&2 \beta\left<	 \w, \K(\Urm) - \K(\u) \right>
	  = \beta \int_{\Ocal}( |\Arm(\Urm)|^2-|\Arm(\u)|^2)^2 \drm x + \beta \int_{\Ocal} |\Arm(\w)|^2( |\Arm(\Urm)|^2+|\Arm(\u)|^2) \drm x.
\end{align}

From Hypothesis \ref{hyp-noise}, we have
	\begin{align}\label{Diff-4}
		\|\Grm(s,\Urm)-\Grm(s,\u)\|^2_{\mathcal{L}_{\Qrm}} \leq L \|\w\|_2^2.
	\end{align}
	Using $\eqref{b0}_2$, \eqref{Sobolev+Korn}, \eqref{poin}, \eqref{eqn-data-inter}, and H\"older's, Ladyzhenskaya's  and Young's inequalities, we get 
	\begin{align}\label{Diff-5}
		2\left|   b( \w, \u, \w )   + \kappa \left( \w,  \Rrm_{\h}(\w) - \w \right)  \right| 
		& \leq  2\left|   b( \w, \w, \u )  \right|  + 2 \left| \kappa \left( \w, \Rrm_{\h}(\w) - \w \right)  \right| 
		\nonumber\\ 
		&  \leq   2  \|\w\|_{\Vbb} \|\w\|_2 \|\u\|_{\infty}   + 2  \kappa \| \w\|_2 \| \Rrm_{\h}(\w) - \w\|_2 
		\nonumber\\ 
		&  \leq    \frac{2 \mathrm{N}_d}{(\lambda_1)^{\frac14}}  \|\w\|^{\frac32}_{\Vbb} \|\w\|_2^{\frac12} \|\Arm(\u)\|_{4}  +   \kappa \| \w\|^2_2 + \kappa \| \Rrm_{\h}(\w) - \w\|^2_2 
		\nonumber\\ 
		&  \leq  \nu\varepsilon_0 \|\w\|^2_{\Vbb} +  \frac{27[\mathrm{N}_d]^4}{16\lambda_1\nu^3\varepsilon_0^3} \|\w\|^2_2 \|\Arm(\u)\|^4_{4} +   \kappa \| \w\|^2_2 + \kappa c_0 \h^2 \|\w\|^2_{\Vbb} .
	\end{align}
	Combining \eqref{Diff-1}-\eqref{Diff-5} yields the following:
	\begin{align}\label{Diff-6}
		& \|\w(t)\|^2_2 + (\nu\varepsilon_0 -\kappa c_0 \h^2) \int_0^t \|\w(s)\|^2_{\Vbb}\,\drm s  +  \int_0^t  \left(\kappa - \frac{27[\mathrm{N}_d]^4}{16\lambda_1\nu^3\varepsilon_0^3} \|\Arm(\u(s))\|_{4}^4  - L\right) \|\w(s)\|^2_2 \,\drm s 
		\nonumber\\
		&  \leq  \|\w(0)\|^2_2 + 2 \int_0^t \big( \w (s), [\Grm(\Urm(s))-\Grm(\u(s))]\, \drm\Wrm(s) \big).	
	\end{align}
	For $0<\kappa \le \frac{\nu \varepsilon_0}{c_0 \h^2}$, applying the It\^o formula to the process $e^{\varrho(\cdot)}\|\w(\cdot)\|_{2}^2$ yields, $\mathbb{P}$-a.s.,
	\begin{align}\label{Diff-7}
		& e^{\varrho(t)} \|\w(t)\|^2_2  + \frac{\kappa\delta}{1+\delta} \int_{0}^{t} e^{\varrho(s)} \|\w(s)\|^2_2 \drm s
		\leq \|\w(0)\|^2_2 + 2 \int_0^t e^{\varrho(s)}  \big( \w (s), [\Grm(\Urm(s))-\Grm(\u(s))]\, \drm\Wrm(s) \big),
	\end{align}
where 
\begin{align*}%\label{eqn-varrho}
	\varrho(t)= \left(\frac{\kappa}{1+\delta} -  L \right)t  -  \frac{27[\mathrm{N}_d]^4}{16\lambda_1\nu^3\varepsilon_0^3} \int_0^t\|\Arm(\u(s))\|_{4}^4\drm s
\end{align*} 
so that 
\begin{align*}
	\varrho'(t)= \frac{\kappa}{1+\delta} -  L - \frac{27[\mathrm{N}_d]^4}{16\lambda_1\nu^3\varepsilon_0^3} \|\Arm(\u(t))\|_{4}^4, \ \text{ for a.e. } t.
\end{align*}
	 Applying expectations to \eqref{Diff-7} and recalling that the stochastic integral is a (local) martingale yields
	\begin{align*}
		& \Ebb \left[e^{\varrho(t\land \tau)} \|\w(t\land \tau)\|^2_2 \right] 
		\leq  \|\w(0)\|^2_2,
	\end{align*}
	for any $\delta\geq 0$, any stopping time $\tau$ and any $t \geq 0$. This conclude the proof.	
\end{proof}

Next, a suitable stopping time is introduced to control the integrating factor in \eqref{eqn-Difference-3}. For $R$, $\uppi$, $\delta$, $\kappa>0,$ define
\begin{align}\label{eqn-stopping-time}
	\uptau_{R,\uppi,\delta,\kappa}:= \inf \left\{ s\geq 0: \frac{27[\mathrm{N}_d]^4}{16\lambda_1\nu^3\varepsilon_0^3}\int_0^s\|\Arm(\u(t))\|^4_{4} \drm t + \left(L - \frac{\kappa}{(1+\delta)^2}\right)s - \uppi \geq R \right\},
\end{align}
and $\uptau_{R,\uppi,\delta,\kappa}=+\infty$ is the set is empty, that is, if 
\begin{align*}
	\frac{27[\mathrm{N}_d]^4}{16\lambda_1\nu^3\varepsilon_0^3}\int_0^s\|\Arm(\u(t))\|^4_{4} \drm t + \left(L - \frac{\kappa}{(1+\delta)^2}\right)s - \uppi < R, \;\;\; \text{ for all } s\geq 0.
\end{align*}
The parameter $\uppi$ is used to capture the dependence on the initial data $\u_0$, the external forcing $\f$, and other relevant constants.

By combining the definition of the stopping time $\uptau_{R,\uppi,\delta,\kappa}$ with Lemma \ref{lem-Difference}, the following result follows immediately.
\begin{corollary}\label{cor-stop-time}
	Under the Hypothesis of Lemma \ref{lem-Difference}, for any $\u_0$, $\Urm_0\in\Hbb$ and any $R$, $\uppi$, $\delta>0$, we get
	\begin{align*}
		\Ebb\left[ \mathbf{1}_{(\uptau_{R,\uppi,\delta,\kappa}=+\infty)}  \|\u(t)-\Urm(t)\|^2_{2} \right] \leq e^{R+\uppi - \frac{\delta\kappa}{(1+\delta)^2}t} \|\u_0-\Urm_0\|^2_2.
	\end{align*}
\end{corollary}
\begin{proof}
	See the proof of \cite[Corollary 4.4]{BFLZ_2025_Arxiv}.
\end{proof}

The result below establishes, for appropriately selected parameters $\uppi$, $\delta$, and $\kappa$, an estimate of $\mathbb{P}(\uptau_{R,\uppi,\delta,\kappa}<+\infty)$ as a function of $R$.

\begin{proposition}\label{prop-prob-esti}
Let $R>0$ and $\delta>0$ be arbitrary given, and condition \eqref{third-grade-paremeters-res} and Hypothesis \ref{hyp-noise} are satisfied. For the solution of system \eqref{STGF} with initial data $\u_0$, we consider the stopping time $\uptau_{R,\uppi,\delta,\kappa}$ defined in \eqref{eqn-stopping-time}. If
		\begin{align}\label{pi-3}
			\uppi  \geq \frac{27[\mathrm{N}_d]^4}{4\beta\lambda_1\nu^3\varepsilon_0^3} \|\u_0\|^2_2
		\end{align} 
		and 
		\begin{align}\label{sigma-3}
		\kappa \geq \left\{ \frac{27[\mathrm{N}_d]^4}{4\beta\lambda_1\nu^3\varepsilon_0^3} \Mrm + L \right\} (1+\delta)^2 
		\end{align} 
	where $\Mrm$ and $\mathrm{N}_d$ are given in \eqref{eqn-Mrm} and \eqref{Sobolev+Korn}, respectively,	then
		\begin{align*}
			\Pbb(\uptau_{R,\uppi,\delta,\kappa}<+\infty)
			\leq 
			\begin{cases}
				e^{-  \frac{\nu\lambda_1}{4K} \cdot \frac{4\beta\lambda_1\nu^3\varepsilon_0^3}{27[\mathrm{N}_d]^4} \cdot R}, & \text{ for } L=0; \\
				e^{- \min \left\{ \frac{\beta\lambda_1}{16L |\Ocal|}, \frac{\nu\lambda_1}{4K} \right\}\cdot \frac{4\beta\lambda_1\nu^3\varepsilon_0^3}{27[\mathrm{N}_d]^4} \cdot  R}, & \text{ for } L>0.
			\end{cases}
		\end{align*} 
\end{proposition}

\begin{proof}
	The proof follows by the similar lines as in the proof of \cite[Proposition 4.3]{BZ_DCDS} and \cite[Proposition 4.5]{BFLZ_2025_Arxiv}. But for the clarity related to the condition on the parameters, we provide the proof. 
	
	Let us introduce the set
	\begin{align}
		\mathfrak{F}_{R,\uppi,\delta,\kappa} := \left\{ \sup_{s \geq 0}\left[ \frac{27[\mathrm{N}_d]^4}{16\lambda_1\nu^3\varepsilon_0^3}\int_0^s\|\Arm(\u(t))\|^4_{4} \drm t + \left(L - \frac{\kappa}{(1+\delta)^2}\right)s - \uppi \right] \geq R  \right\},
	\end{align}
	so that $\Pbb(\uptau_{R,\uppi,\delta,\kappa} < \infty) \leq \Pbb (\mathfrak{F}_{R,\uppi,\delta,\kappa})$. The compliment of the set is given by
%	\begin{align}
%		\mathfrak{F}^c_{R,\uppi,\delta,\kappa} = \left\{ \frac{27[\mathrm{N}_d]^4}{16\lambda_1\nu^3\varepsilon_0^3} \int_0^s\|\Arm(\u(t))\|^4_{4} \drm t + \left(L - \frac{\kappa}{(1+\delta)^2}\right)s - \uppi < R, \;\;\; \text{ for all } r\geq 0\right\}.
%	\end{align}
	\begin{align}
		\mathfrak{F}^c_{R,\uppi,\delta,\kappa} = \left\{ \frac{\beta}{4}\int_0^s\|\Arm(\u(t))\|^4_{4} \drm t <  \frac{4\beta\lambda_1\nu^3\varepsilon_0^3}{27[\mathrm{N}_d]^4} \left[\left( \frac{\kappa}{(1+\delta)^2} - L \right)s + \uppi + R\right], \;\;\; \text{ for all } s\geq 0\right\}.
	\end{align}
From the assumption \eqref{sigma-3}, we have 
	\begin{align}
		 %K  + \frac{1}{4\beta}\left(\frac{\widetilde{K}}{\lambda_1}+1\right)^2|\Ocal|  +  \frac{108}{\beta^3} |\Ocal| |\alpha|^4 	+  \|\f\|^2_{\Vbb^{*}} = 
		 \Mrm \leq  \frac{4\beta\lambda_1\nu^3\varepsilon_0^3}{27[\mathrm{N}_d]^4} \left( \frac{\kappa}{(1+\delta)^2} - L \right).
	\end{align}
Let us choose $\uppi>0$ as in \eqref{pi-3} and take $\bar{R} =  \frac{4\beta\lambda_1\nu^3\varepsilon_0^3}{27[\mathrm{N}_d]^4} R$. Choosing these parameters results in
\begin{align*}
	\mathfrak{F}^c_{R,\uppi,\delta,\kappa} \supseteq \left\{ \frac{\beta}{4}\int_0^s\|\Arm(\u(t))\|^4_{4} \drm t <  \Mrm t + \|\u_0\|^2_2 + \bar{R}, \;\;\; \text{ for all } s\geq 0\right\},
\end{align*}
that is,
\begin{align*}
	\mathfrak{F}_{R,\uppi,\delta,\kappa} \subseteq \left\{ \sup_{s \geq 0} \left[ \frac{\beta}{4}\int_0^s\|\Arm(\u(t))\|^4_{4} \drm t -  \Mrm t  - \|\u_0\|^2_2\right] \geq \bar{R} \right\}.
\end{align*}
By the application of Lemma \ref{lem-prob-esti-2}, we infer that 
\begin{align*}
	\Pbb (\mathfrak{F}_{R,\uppi,\delta,\kappa}) \leq \begin{cases}
		e^{-  \frac{\nu\lambda_1}{4K} \cdot \frac{4\beta\lambda_1\nu^3\varepsilon_0^3}{27[\mathrm{N}_d]^4} \cdot R}, & \text{ for } L=0; \\
		e^{- \min \left\{ \frac{\beta\lambda_1}{16L |\Ocal|}, \frac{\nu\lambda_1}{4K} \right\}\cdot \frac{4\beta\lambda_1\nu^3\varepsilon_0^3}{27[\mathrm{N}_d]^4} \cdot  R}, & \text{ for } L>0.
	\end{cases}
\end{align*}
Hence, in view of the fact that $\Pbb(\uptau_{R,\uppi,\delta,\kappa} < \infty) \leq \Pbb (\mathfrak{F}_{R,\uppi,\delta,\kappa})$, we complete the proof.
\end{proof}

In what follows, we state and prove the main result of this section.

\begin{theorem}\label{MT-Critical}
	Suppose that condition \eqref{third-grade-paremeters-res},  Hypothesis \ref{hyp-noise} and assumption \eqref{eqn-data-inter} are satisfied. Assume that $\u$  and $\Urm$ are the solutions to the systems \eqref{STGF} and \eqref{STGF-CDA}, respectively. If $\kappa$ and $\h$ are such that 
		\begin{equation}\label{condition-on-sigma-6}
			  \frac{27[\mathrm{N}_d]^4}{4\beta\lambda_1\nu^3\varepsilon_0^3} \Mrm 	 + L  <  \kappa  \leq   \frac{  \nu \varepsilon_0 }{c_0 \h^2 }, 
		\end{equation}
		then $\mathbb{E}\left[\|\u(t)-\Urm(t)\|_2^2\right] \to 0$ exponentially fast as $t \to + \infty$, where $\Mrm$ is given by \eqref{eqn-Mrm}.
\end{theorem}
\begin{proof}
	Making use of H\"older's inequality, Corollary \ref{cor-stop-time} and Lemmas \ref{lem-pth-moments}-\ref{lem-pth-moments-Urm}, we find, for any $R, \uppi, \delta>0$ and $0<\kappa \leq \frac{\nu\varepsilon_0}{c_0 \h^2 }$, 
	\begin{align}\label{MT-1-1}
		\mathbb{E} \left[ \|\u(t)-\Urm(t)\|^2_2\right]
		& = \mathbb{E} \left[\mathbf{1}_{(\uptau_{R, \uppi, \delta, \kappa =+\infty  })} \|\u(t)-\Urm(t)\|^2_2\right] + \mathbb{E} \left[\mathbf{1}_{(\uptau_{R, \uppi, \delta, \kappa <+\infty})} \|\u(t)-\Urm(t)\|^2_2\right]
		\notag \\
		& \leq e^{ R + \uppi - \frac{\delta \kappa}{(1+\delta)^2} t} \|\u_0-\Urm_0\|^2_2 + \left( \mathbb{P}(\uptau_{R, \uppi, \delta, \kappa}<+\infty)\right)^{\frac12} \left(\mathbb{E}\left[\|\u(t)-\Urm(t)\|^4_2\right] \right)^{\frac12}
		\notag\\
		& \leq
		C\left(1+\|\u_0\|^2_2 +  \|\Urm_0\|^2_2 \right)\left(\left( \mathbb{P}(\uptau_{R, \uppi, \delta, \kappa}<+\infty)\right)^{\frac12}+ e^{ R + \uppi - \frac{\delta\kappa}{(1+\delta)^2} t}\right),
	\end{align}
	where $C>0$ is a constant independent of $t$. We proceed by fixing suitable parameters $R$, $\uppi$, $\delta$, and $\kappa$, and then exploiting the previously established bounds for $\uptau_{R,\uppi,\delta,\kappa}$. From the Proposition \ref{prop-prob-esti}, we gain
	\begin{align}\label{MT-1-2}
		(\mathbb{P}(\uptau_{R, \uppi, \delta, \kappa}<+\infty))^{\frac12} 
		\leq 
		\begin{cases}
			e^{-C_1 R}, & \text{ for } L=0;\\
			e^{-C_2 R}, & \text{ for } L>0,
		\end{cases} 
	\end{align}
	where the constants $C_1>0$ and $C_2>0$ are independent of $R$ and $t$. Now combining \eqref{MT-1-1} and \eqref{MT-1-2}, and choosing $R= \frac{\delta\kappa}{2(1+\delta)^2}t$, for each $t > 0$, we immediately complete the proof. 
\end{proof}

\section{Pathwise convergence}\label{Sec5} 

In Theorem \ref{MT-Critical}, we established that $\lim\limits_{t \to +\infty} \Ebb \left[ \|\u(t)-\Urm(t)\|_2^2 \right] = 0.$ As a consequence, there exists a subsequence $\{t_{n}\}_{n\in\N}$ such that $t_n \to +\infty$ as $n \to +\infty$, along which $\|\u(t)-\Urm(t)\|_2^2$ converges $\Pbb$-a.s. Moreover, as shown in \cite[Corollary 4.6]{BFLZ_2025_Arxiv}, for any sequence $\{t_n\}_{n\in\mathbb{N}}$ with $t_n \to +\infty$ as $n \to +\infty$, one can obtain that $ \|\u(t_n)-\Urm(t_n) \|_2^2 \to 0$ $\Pbb$-a.s. We emphasize, however, that this result alone is not sufficient to conclude pathwise convergence, that is,
\[
\lim\limits_{t \to +\infty} \|\u(t)-\Urm(t) \|_2^2 =0,\ \Pbb\text{-a.s.}
\]

\begin{corollary}
\label{pathwise_mult}
Under the same assumption as in Theorem \ref{MT-Critical}, for any sequence of times $\{t_{n}\}_{n\in\N}$ such that $t_n \to +\infty$ as $n \to +\infty$, we get 
 \[
\mathbb{P}\left( \lim_{ n \to + \infty}\|\u(t_n)-\Urm(t_n)\|^2_2 =0\right) = 1.
 \]
\end{corollary}
\begin{proof}
See the proof of \cite[Corollary 4.6]{BFLZ_2025_Arxiv}.
\end{proof}

We now turn to the pathwise analysis, in which convergence results are understood in the $\Pbb$-a.s. sense. Throughout this part, we assume that $\Grm$ is independent of $\u$.

Let us first provide an auxiliary lemma which will be used in sequel.

\begin{lemma}
\label{data-estimate-1}
Let $\Grm\in \mathcal{L}_{\Qrm}(\Hbb)$ and let $\Rrm_{\h}$ satisfy \eqref{eqn-data-inter}. 
Let $\u$ and $\Urm$ be the solutions to the systems \eqref{STGF} and \eqref{STGF-CDA}, respectively. If $\kappa$ and $\h$ satisfy
\begin{equation}
\label{condi-1-mu}
0<\kappa \leq \frac{  \nu \varepsilon_0 }{c_0 \h^2 } 
\end{equation}
 where $c_0$ is the constant that appears in estimate \eqref{eqn-data-inter}, then $\Pbb$-a.s.  we have
\begin{align}\label{error-u-U}
\left\| \u(t)-\Urm(t) \right\|_{2}^2 &\leq \left\|\u_0 - \Urm_0 \right\|_{2}^2 \cdot  \exp \left(-\kappa t + \displaystyle \frac{27[\mathrm{N}_d]^4}{16\lambda_1\nu^3\varepsilon_0^3} \int_{0}^t   \|\Arm(\u(s))\|_{4}^4 \drm s   \right), \qquad \text{ for any }t>0.
\end{align} 
\end{lemma}
\begin{proof}
	Since $\Grm$ is independent of $\u$ and consequently $L=0$, we infer from \eqref{Diff-6} that $\w:= \u-\Urm$ satisfies
	\begin{align}\label{Diff-13}
		& \|\w(t)\|^2_2 + (\nu\varepsilon_0 -\kappa c_0 \h^2) \int_0^t \|\w(s)\|^2_{\Vbb}\,\drm s  +  \int_0^t  \left(\kappa - \frac{27[\mathrm{N}_d]^4}{16\lambda_1\nu^3\varepsilon_0^3} \|\Arm(\u(s))\|_{4}^4  \right) \|\w(s)\|^2_2 \,\drm s 
		  \leq  \|\w(0)\|^2_2.
	\end{align}
Hence, an application of Gronwall's inequality completes the proof.
\end{proof}

Now, we establish the main result of this subsection.

\begin{theorem}\label{pathwise_data_ass}
	Let $\Grm\in \mathcal{L}_{\Qrm}(\Hbb)$ and let $\Rrm_{\h}$ satisfy \eqref{eqn-data-inter}. Let $\u$ and $\Urm$ be the solutions to the systems \eqref{STGF} and \eqref{STGF-CDA}, respectively. If $\kappa$ and $\h$ satisfy
	\begin{equation}\label{cond_mu}
		\frac{27[\mathrm{N}_d]^4}{4\beta\lambda_1\nu^3\varepsilon_0^3} \left[\|\Grm\|^2_{\mathcal{L}_{\Qrm}}  + \left(\frac{1}{4\beta}  +  \frac{27\alpha^4 }{4\beta^3}\right) |\Ocal| 
		+  \|\f\|^2_{\Vbb^{*}}\right] <   \kappa \leq   \frac{  \nu \varepsilon_0 }{c_0 \h^2 }, 
	\end{equation}
	then
	$\|\u(t)-\Urm(t)\|_2 \to 0$ exponentially fast as $t \to + \infty$, $\mathbb{P}$-a.s.
\end{theorem}

\begin{proof}
In view of \eqref{eqn-Prob-Est-varpi=3}, we have 
\begin{align}
	& \Pbb\left\{\sup_{t \geq T} \left( \|\u(t)\|^{2}_{2} + \frac{\beta}{4} \int_{0}^{t}\|\Arm(\u(s))\|^4_{4} \drm s - \|\u_0\|^{2}_{2}  - \left[\|\Grm\|^2_{\mathcal{L}_{\Qrm}}  + \left(\frac{1}{4\beta}  +  \frac{27\alpha^4 }{4\beta^3}\right) |\Ocal| 
	+  \|\f\|^2_{\Vbb^{*}}\right] t  \right) \geq R \right\}
\nonumber\\ 
& 	\leq 
		e^{-  \frac{\nu\lambda_1}{4K} R}, 
\end{align}
for all $T\geq 0$ and $R>0$, which implies
\begin{align}
	& \Pbb\left\{ \limsup_{t \to +\infty} \frac{1}{t} \frac{27[\mathrm{N}_d]^4}{16\lambda_1\nu^3\varepsilon_0^3} \int_{0}^{t}\|\Arm(\u(s))\|^4_{4} \drm s  \leq   \frac{27[\mathrm{N}_d]^4}{4\beta\lambda_1\nu^3\varepsilon_0^3} \left[\|\Grm\|^2_{\mathcal{L}_{\Qrm}}  + \left(\frac{1}{4\beta}  +  \frac{27\alpha^4 }{4\beta^3}\right) |\Ocal| 
	+  \|\f\|^2_{\Vbb^{*}}\right]  \right\} = 1.
\end{align}
Therefore, for any initial velocity $\u_0\in \Hbb$, 
\begin{align}\label{limit-t}
	\limsup_{t \to +\infty} \frac{1}{t} \frac{27[\mathrm{N}_d]^4}{16\lambda_1\nu^3\varepsilon_0^3} \int_{0}^{t}\|\Arm(\u(s))\|^4_{4} \drm s  \leq   \frac{27[\mathrm{N}_d]^4}{4\beta\lambda_1\nu^3\varepsilon_0^3} \left[\|\Grm\|^2_{\mathcal{L}_{\Qrm}}  + \left(\frac{1}{4\beta}  +  \frac{27\alpha^4 }{4\beta^3}\right) |\Ocal| 
	+  \|\f\|^2_{\Vbb^{*}}\right], \;\;\; \Pbb\text{-a.s.}
\end{align}
From  \eqref{cond_mu} there exists a constant $\kappa^*>0$ such that
\[
 \kappa-  \frac{27[\mathrm{N}_d]^4}{4\beta\lambda_1\nu^3\varepsilon_0^3} \left[\|\Grm\|^2_{\mathcal{L}_{\Qrm}}  + \left(\frac{1}{4\beta}  +  \frac{27\alpha^4 }{4\beta^3}\right) |\Ocal| 
 +  \|\f\|^2_{\Vbb^{*}}\right]  - \kappa^* >0.
\]
This gives along with \eqref{limit-t}
\begin{align}
	& \lim_{t \to +\infty}\exp \left(-\kappa t + \frac{27[\mathrm{N}_d]^4}{16\lambda_1\nu^3\varepsilon_0^3} \int_{0}^{t}\|\Arm(\u(s))\|^4_{4} \drm s   \right)
	\nonumber\\
	&  \leq
	\lim_{t \to +\infty}\exp \bigg( - \kappa^* t - \frac{27[\mathrm{N}_d]^4}{4\beta\lambda_1\nu^3\varepsilon_0^3} \left[\|\Grm\|^2_{\mathcal{L}_{\Qrm}}  + \left(\frac{1}{4\beta}  +  \frac{27\alpha^4 }{4\beta^3}\right) |\Ocal| 
	+  \|\f\|^2_{\Vbb^{*}}\right]  t 
	\nonumber\\ 
	& \qquad \qquad \qquad + \frac{27[\mathrm{N}_d]^4}{16\lambda_1\nu^3\varepsilon_0^3} \int_{0}^{t}\|\Arm(\u(s))\|^4_{4} \drm s   \bigg)  =0,
\end{align}
and by \eqref{error-u-U} we also have that $\u-\Urm$ vanishes for large times. 
Particularly,  we see that there exists a time $t_0>0$ such that for any $t>t_0$, we have
\[
\|\u(t)-\Urm(t)\|_2^2 \leq 
\|\u_0-\Urm_0\|^2_2
  \exp\left(-\left(\kappa- \frac{27[\mathrm{N}_d]^4}{4\beta\lambda_1\nu^3\varepsilon_0^3} \left[\|\Grm\|^2_{\mathcal{L}_{\Qrm}}  + \left(\frac{1}{4\beta}  +  \frac{27\alpha^4 }{4\beta^3}\right) |\Ocal| 
  +  \|\f\|^2_{\Vbb^{*}}\right]  - \kappa^* \right)t\right),
\]
showing the exponential convergence rate. This completes the proof.
\end{proof}

%
%\medskip\noindent
%{\bf Acknowledgments:}     This work is funded by national funds through the FCT - Fundação para a Ciência e a Tecnologia, I.P., under the scope of the projects UID/297/2025 and UID/PRR/297/2025 (Center for Mathematics and Applications - NOVA Math).  K. Kinra wish to thank Prof. Fernanda Cipriano for suggesting this problem.

\medskip\noindent
\textbf{Data availability:} No data was used for the research described in the article.

\medskip\noindent
\textbf{Declarations}: During the preparation of this work, the authors have not used AI tools.

\medskip\noindent
\textbf{Author Contributions}: The sole author wrote, and edited the entire manuscript.

\medskip\noindent
\textbf{Conflict of interest:} The author declares no conflict of interest.

\appendix

%\printbibliography


\begin{thebibliography}{99}



\bibitem{azouani2014continuous} A. Azouani, E. Olson and E.S. Titi, Continuous data assimilation using general interpolant observables, {\em Journal of Nonlinear Science}, \textbf{24} (2014), 277-304.

\bibitem{AbderrahimTiti2014} A. Azouani and E.S. Titi, Feedback control of nonlinear dissipative systems by finite
  determining parameters - a reaction-diffusion paradigm, {\em Evolution Equations and Control Theory}, \textbf{3}(4) (2014), 579-594.



\bibitem{BB_RMS_2025}  A. Balakrishna and A. Biswas, Determining functionals and data assimilation and a novel regularity criterion for the three-dimensional Navier-Stokes equations, \emph{Res. Math. Sci.}, {\bf 12}(3) (2025), Paper No. 46, 29 pp.  %; MR4928012



\bibitem{BFLZ_2025_Arxiv}  H. Bessaih, B. Ferrario, O. Landoulsi and M. Zanella, Continuous data assimilation for 2D stochastic Navier-Stokes equations. \url{https://arxiv.org/pdf/2512.15184v1}


\bibitem{bessaih2022continuous}
H. Bessaih, V. Ginting and B. McCaskill, Continuous data assimilation for displacement in a porous medium, {\em Numerische Mathematik}, \textbf{151}(4) (2022), 927-962.




\bibitem{bessaih2015continuous}
H. Bessaih, E. Olson and E.S. Titi, Continuous data assimilation with stochastically noisy data, {\em Nonlinearity}, \textbf{28}(3) (2015).

%\bibitem{BHLP} A. Biswas, J. Hudson, A. Larios and Y. Pei, Continuous data assimilation for the 2{D} magnetohydrodynamic equations using one component of the velocity and magnetic fields, {\em Asymptot. Anal.}, \textbf{108}(1-2) (2018), 1-43.



\bibitem{BP_SIMA_2021}  A. Biswas and R. Price, Continuous data assimilation for the three-dimensional Navier-Stokes equations, \emph{SIAM J. Math. Anal.}, {\bf 53}(6) (2021), 6697-6723. %; MR4344886


\bibitem{Blomker} D. Bl\"omker, K. Law, A.M. Stuart, and K.C. Zygalakis, Accuracy and stability of the continuous-time 3DVAR filter for the Navier--Stokes equation, {\em Nonlinearity}, \textbf{26}(8) (2013), 2193-2219.
 

\bibitem{CGJA} Y. Cao, A. Giorgini, M. Jolly and A. Pakzad, Continuous data assimilation for the 3{D} {L}adyzhenskaya model: analysis and computations, {\em Nonlinear Anal. Real World Appl.}, \textbf{68} (2022), Paper No. 103659.
  
  
  
  \bibitem{foias1991determining} F. Ciprian and E.S. Titi, Determining nodes, finite difference schemes and inertial manifolds, {\em Nonlinearity}, \textbf{4}(1) (1991).
  

\bibitem{daley1993atmospheric} R. Daley, {\em Atmospheric data analysis}, Cambridge University Press, 1993.



	\bibitem{DaZ}	G. Da Prato and J. Zabczyk, \emph{Stochastic Equations in Infinite Dimensions}, Cambridge University Press, 2014.

 
 \bibitem{DiFratta+Solombrino_Arxiv} G. Di Fratta and F. Solombrino, Korn and Poincar\'e-Korn inequalities: a different perspective, \textit{Proc. Amer. Math. Soc.} {\bf 153}(1) (2025), 143--159.%; MR4840265

%\bibitem{DSZ_2025_Arxiv} A. Di Primio, L. Scarpa, and M. Zanella, Existence, uniqueness and asymptotic stability of invariant measures for the stochastic Allen-Cahn-Navier-Stokes system with singular potential, \url{https://arxiv.org/pdf/2501.06174}.





\bibitem{FGHMMW} A.~Farhat, N.~E. Glatt-Holtz, V.~R. Martinez, S.~A. McQuarrie, and J.~P.
  Whitehead, Data assimilation in large {P}randtl {R}ayleigh-{B}\'{e}nard
  convection from thermal measurements, {\em SIAM J. Appl. Dyn. Syst.}, \textbf{19}(1) (2020), 510-540.





%\bibitem{RFHK} R. Farwig,  H. Kozono and  H. Sohr, An $L^q$-approach to Stokes and Navier-Stokes equations in general domains, \emph{Acta Math.} {\bf 195} (2005), 21-53. 





\bibitem{Farwig+Kozono+Sohr_2007}  R. Farwig, H. Kozono and H. Sohr, On the Helmholtz decomposition in general unbounded domains, \textit{Arch. Math. (Basel)}, {\bf 88}(3) (2007), 239-248.  %; MR2305602




 


\bibitem{BZ_DCDS} B. Ferrario and M. Zanella, Uniqueness of the invariant measure and asymptotic stability for the 2D Navier-Stokes equations with multiplicative noise, \emph{Discrete Contin. Dyn. Syst.}, {\bf 44}(1) (2024), 228-262. % ; MR4671522


\bibitem{FZ25} B. Ferrario and M. Zanella, Long time behavior of the stochastic 2{D} {N}avier-{S}tokes
  equations, {\em Commun. Math. Anal. Appl.}, \textbf{4}(4) (2025), 550-576.

 
 
 \bibitem{FR80} R.L. Fosdick and , K.R. Rajagopal, Thermodynamics and stability of fluids of third grade, \textit{Proc. Roy. Soc. London Ser. A},  \textbf{339} (1980), 351-377.
 
 
 
% \bibitem{DFHM} D. Fujiwara, H. Morimoto, An $L^r$-theorem of the Helmholtz decomposition of vector fields, \emph{J. Fac. Sci. Univ. Tokyo Sect. IA Math.,} {\bf 24}(1977),	685-700.
 
 
 
 
% \bibitem{GKM_AMOP} S. Gautam, K. Kinra and M.T. Mohan, Feedback stabilization of convective Brinkman-Forchheimer extended Darcy equations, \emph{Appl. Math. Optim.}, \textbf{91}(1) (2025), Paper No. 25.
 
 
% \bibitem{GM_2025} S. Gautam and M.T. Mohan, 	On convective Brinkman-Forchheimer equations, \emph{Dyn. Partial Differ. Equ.}, \textbf{22}(3) (2025), pp. 191-233.
% 
 
 
 
 
 
 
 
 
 
 \bibitem{Glatt-Holtz_2014_Arxiv}  N. Glatt-Holtz, Notes on statistically invariant states in stochastically driven fluid flows, \url{https://arxiv.org/pdf/1410.8622}.
 
 
 

\bibitem{Hamza+Paicu_2007} M. Hamza and M. Paicu, {Global existence and uniqueness result of a class of third-grade fluids equations,} \textit{Nonlinearity}, \textbf{20}(5)	(2007), 1095-1114.


 

%\bibitem{GHMR17} N.~Glatt-Holtz, J.C.~Mattingly and G.~Richards, On unique ergodicity in nonlinear stochastic partial differential equations, {\em Journal of Statistical Physics}, \textbf{166} (3-4), (2017), 618-649. 






\bibitem{JST} M.S. Jolly, T. Sadigov and E.S. Titi, Determining form and data assimilation algorithm for weakly damped
  and driven {K}orteweg-de {V}ries equation-{F}ourier modes case, {\em Nonlinear Anal. Real World Appl.}, \textbf{36} (2017), 287-317.

\bibitem{jones1992determining} D.A. Jones and E.S. Titi, Determining finite volume elements for the 2d Navier-Stokes
  equations, {\em Physica D: Nonlinear Phenomena}, \textbf{60}(1-4) (1992), 165-174.

\bibitem{jones1993upper} D.A Jones and E.S. Titi, Upper bounds on the number of determining modes, nodes, and volume elements for the Navier-Stokes equations, {\em Indiana University Mathematics Journal}, \textbf{42}(3) (1993), 875-887.




\bibitem{Kesavan_1989}  S. Kesavan, \textit{Topics in functional analysis and applications},
John Wiley \& Sons, Inc., New York,  1989.





\bibitem{KK_CDA_SCBF} K. Kinra, A note on continuous data assimilation for stochastic convective Brinkman-Forchheimer equations in 2D and 3D, (2026). \url{https://arxiv.org/pdf/2601.17650}


%\bibitem{KKMTM-DCDSB} {K. Kinra} and M.T. Mohan, Random attractors and invariant measures for stochastic convective Brinkman-Forchheimer equations on 2D and 3D unbounded domains, \emph{Discrete Contin. Dyn. Syst. Ser. B}, \textbf{29}(1) (2024), 377-425. 

%\bibitem{KK+MTM-SCBF} K. Kinra and M.T. Mohan, Stochastic convective Brinkman-Forchheimer equations on general unbounded domains, \url{https://arxiv.org/pdf/2007.09376}.
%
%\bibitem{KK+FC+MTM-SCBF} K. Kinra, F. Cipriano and M.T. Mohan, Martingale solution, invariant measure and ergodicity for stochastic convective Brinkman-Forchheimer equations on general domains in $\mathbb{R}^d$, \url{https://arxiv.org/pdf/2109.05510v2}.
%



%\bibitem{HKTY}   H. Kozono and T. Yanagisawa, $L^r$-variational inequality for vector fields and the Helmholtz-Weyl decomposition in bounded domains, \emph{Indiana Univ. Math. J.}, {\bf  58}(4) (2009), 1853-1920.


%\bibitem{Krylov_2022}  Nikolai Krylov, \emph{Introduction to the theory of random processes}, \textbf{Vol. 43}, American Mathematical Soc., 2002.


%\bibitem{KS} A.~Kulik and M.~Scheutzow, Generalized couplings and convergence of transition probabilities, {\em Probability Theory and Related Fields}, \textbf{171}(1-2) (2018), 333-376.

\bibitem{Kunstmann_2010}  P.C. Kunstmann, Navier-Stokes equations on unbounded domains with rough initial data, {\it Czechoslovak Math. J.}, {\bf 60(135)}(2) (2010), 297-313. %; MR2657950



%\bibitem{LP} A. Larios and Y. Pei, Approximate continuous data assimilation of the 2{D} {N}avier-{S}tokes equations via the {V}oigt-regularization with observable data, {\em Evol. Equ. Control Theory}, \textbf{9}(3) (2020), 733-751.

 



\bibitem{PAM}	P.A. Markowich, E.S. Titi and S. Trabelsi,	Continuous data assimilation for the three-dimensional Brinkman-Forchheimer-extended Darcy model, \emph{Nonlinearity}, {\bf 29}(4) (2016), 1292-1328.









%\bibitem{MTM9} 	M. T. Mohan,   $\mathbb{L}^p$-solutions of deterministic and stochastic convective Brinkman-Forchheimer equations, \emph{Analysis and Mathematical Physics}, {\bf 11} (2021), Ar. No.: 164. 
%
%
%
%
%
%\bibitem{MTM6} M. T. Mohan,  Well-posedness and asymptotic behavior of stochastic convective Brinkman-Forchheimer equations perturbed by pure jump noise, \emph{Stoch PDE: Anal Comp}, {\bf 10}(2) (2022), 614-690.%; MR4439993. 


\bibitem{NKC_2025_Submit}  R. Nouira, K. Kinra and F. Cipriano, Continuous data assimilation for a class of non-Newtonian fluids of differential type in 2D and 3D, \emph{J. Math. Anal. Appl.}, (2026), In press.


	\bibitem{PP19} 	M. Parida and  S. Padhy, Electro-osmotic flow of a third-grade fluid past a channel having stretching walls, \textit{Nonlinear Eng.}, \textbf{8}(1) (2019), 56-64.

\bibitem{RHK18}	 G.J. Reddy, A. Hiremath, M. Kumar, Computational modeling of unsteady third-grade fluid flow over a vertical cylinder: A study of heat transfer visualization, \textit{Results Phys.}, \textbf{8}  (2018), 671-682.



\bibitem{Temam_1984} R. Temam, \emph{Navier-Stokes Equations, Theory and Numerical Analysis}, North-Holland, Amsterdam, 1984.




\bibitem{yas-fer_JNS}  Y. Tahraoui and F. Cipriano, Invariant measures for a class of stochastic third-grade fluid equations in 2D and 3D bounded domains, \emph{J. Nonlinear Sci.}, \textbf{34}(6) (2024), Paper No. 107, 42 pp.%; MR4799052


%\bibitem{TV_DCDS-S_2025}  E.S. Titi, C. Victor, On the inadequacy of nudging data assimilation algorithms for non-dissipative systems. \emph{Discrete Contin. Dyn. Syst. Ser. S}, (2025). \url{10.3934/dcdss.2025151}
%
%
%
%
%
%\bibitem{Zhou_2012} Y. Zhou, Regularity and uniqueness for the 3D incompressible Navier–Stokes equations with damping, \emph{Applied Mathematics Letters}, \textbf{25} (2012), 1822-1825.


\end{thebibliography}
\end{document}